\documentclass[psamsfonts]{amsart}

\usepackage{amssymb,amsfonts}
\usepackage[all,arc]{xy}
\usepackage{enumerate}
\usepackage{mathrsfs}
\usepackage{graphicx}

\usepackage[mathscr]{euscript}

\usepackage{orcidlink}

\usepackage{booktabs}
\usepackage{threeparttable}

\newtheorem{theorem}{Theorem}[section]

\newtheorem{proposition}[theorem]{Proposition}
\newtheorem{lemma}[theorem]{Lemma}

\theoremstyle{definition}
\newtheorem{definition}[theorem]{Definition}

\newtheorem{example}[theorem]{Example}
\newtheorem{examples}[theorem]{Examples}

\newtheorem{remark}[theorem]{Remark}

\DeclareMathOperator{\ord}{ord}
\DeclareMathOperator{\scl}{scl}
\newcommand{\Id}{\operatorname{Id}}

\numberwithin{equation}{section}

\begin{document}
	
\title[On the subtractive closure operator in hemirings]{On the distributivity of the subtractive closure operator over ideal operations in hemirings}
		
\author{Peyman Nasehpour\orcidlink{0000-0001-6625-364X}}
	
\address{Peyman Nasehpour\\
Education Department\\
The New York Academy of Sciences\\
New York, NY, USA}
\email{nasehpour@gmail.com}


	
\maketitle
	
\begin{abstract}
	In this paper, we investigate the subtractive closure operator in commutative hemiring theory. While this operator satisfies standard closure properties, it does not distribute over ideal addition, intersection, or multiplication in general. We introduce Kuratowski hemirings (where closure preserves joins) and nucleus hemirings (where closure preserves meets), providing explicit counterexamples and positive results for each. Key findings include: the standard semiring $\mathbb{N}_0$ and the max-plus algebra are nucleus hemirings with multiplicative distributivity; the interval semiring $T = [0,1]$ is Kuratowski and multiplicatively distributive but not nucleus; and polynomial hemirings $S_0[X,Y,Z]$ over zerosumfree hemirings form a broad class of non-nucleus counterexamples. We also correct a recent claim regarding distributivity over arbitrary intersections.
\end{abstract}
	
\section{Introduction}

Subtractive ideals are fundamental objects in the multiplicative ideal theory of hemirings, playing a pivotal role in extending classical results from commutative ring theory to the hemiring setting (see, for example, \cite{Golan1999, Nasehpour2016, Nasehpour2019, Nasehpour2018}). Recall that an ideal $I$ of a hemiring $H$ is called \emph{subtractive} if it satisfies the following condition:
\[
a+b, a \in I \implies b \in I, \qquad \forall\, a,b \in H.
\]
Accordingly, characterizing the smallest subtractive ideal of a hemiring $H$ containing a subset $X \subseteq H$---the subtractive closure of $X$---is of central importance to this structural program. For any subset $X$ of a hemiring $H$, its subtractive closure, denoted by $\scl_H(X)$, is defined as the smallest subtractive ideal containing $X$. Although classic properties such as extensivity, idempotence, and monotonicity are inherited from standard closure operators (cf. Theorem \ref{subtractiveclosureproperties}), the subtractive closure operator fails, in general, to distribute over ideal addition, intersection, or multiplication. The primary objective of this paper is to provide explicit counterexamples illustrating these distributive failures and to delineate the precise structural conditions under which distributivity is preserved.
	
Since the terminology in hemiring theory is not yet fully standardized \cite[p.~3]{Glazek2002}, we begin by fixing our notation and terminology. Furthermore, we briefly summarize the main results obtained herein.
	
A bimagma $(H,+, \cdot)$ is called a hemiring if $(H,+)$ is a commutative monoid with identity element $0$, $(H, \cdot)$ is a semigroup, multiplication distributes over addition from both sides, and $0$ is multiplicatively absorbing (i.e., $h0 = 0h = 0$ for all $h \in H$). A hemiring $H$ is commutative if $ab = ba$ for all $a,b \in H$; throughout this paper, all hemirings are assumed to be commutative. A hemiring is called a semiring if it possesses a multiplicative identity element $1 \neq 0$ \cite[p.~1]{Golan1999}. A hemiring $H$ is zerosumfree if $a+b=0$ implies $a=b=0$ for all $a,b \in H$, and entire if $ab=0$ implies $a=0$ or $b=0$ for all $a,b \in H$. A semiring that is both zerosumfree and entire is called an information algebra \cite[p.~4]{Golan1999}.

A nonempty subset $I$ of a hemiring $H$ is called an ideal (denoted by $I \trianglelefteq H$) if $a+b \in I$ and $ha \in I$ for all $a,b \in I$ and $h \in H$ \cite[p.~65]{Golan1999}. The set of all ideals of a hemiring $H$ is denoted by $\Id(H)$. An ideal $I$ is proper if $I \neq H$, and subtractive if $x+y \in I$ and $x \in I$ imply $y \in I$ for all $x,y \in H$ \cite[p.~66]{Golan1999}.
	
In \S\ref{sec:subtractiveclosure}, we recall closure operators and verify that the subtractive closure $\scl_H$ satisfies the standard axioms. We also obtain the explicit description $\scl_H(I)=\{h \in H : h+i \in I \text{ for some } i \in I\}$ and establish that $\scl_H$ is super-distributive over addition and multiplication, but sub-distributive over intersection (Theorem \ref{subtractiveclosureproperties}).
	
The structural core of our investigation begins in \S\ref{sec:kuratowski}, where we address the failure of full distributivity of subtractive closure over ideal addition by introducing and discussing the class of Kuratowski hemirings---those for which the subtractive closure operator behaves as a finite lattice dilation. To demonstrate that this distributive property is highly non-trivial, we establish a sharp structural dichotomy: it fails even for the prototypical standard semiring $\mathbb{N}_0$, whereas it holds for the ``continuous interval'' semiring $T = [0,1]$ introduced in \cite{Nasehpour2016} (cf. Theorem \ref{NasehpourTsemiring}). Beyond specific examples, we identify broad algebraic environments that are inherently Kuratowski, including subtractive, austere, or uniserial hemirings (Propositions \ref{austereisKuratowski} and \ref{uniserialisKuratowski}). Finally, to explore the behavior of this operator on the opposite extreme of fully subtractive settings, we define porous hemirings and prove that $T$ serves as an explicit archetype of a semiring that is simultaneously porous and Kuratowski (Theorem \ref{aporousKuratowskisemiring}).
	
The focus shifts to intersections of ideals in \S\ref{sec:nucleus}, where we formalize the behavior of subtractive closures over intersections by introducing the class of nucleus hemirings. We establish that intersection-distributivity is not a universal property; specifically, Theorem \ref{nonnucleusthm} provides an explicit counterexample using the interval semiring $T = [0,1]$, while Appendix \ref{appendix:polynomial} constructs an independent counterexample via graded polynomial hemirings. Crucially, these parallel constructions rectify a recent assertion by Goswami and Dube \cite[Lemma 2.1(7)]{GoswamiDube2024} that the subtractive closure operator distributes over arbitrary intersections, showing instead that equality can fail even in finite settings. To delineate the boundaries where this property remains valid, we prove that several broad structural environments---including subtractive, austere uniform, and uniserial hemirings---are inherently nucleus (Propositions \ref{subtractiveisnucleus}, \ref{austereuniformisnucleus}, and \ref{uniserialisnucleus}). Finally, we demonstrate that the foundational canonical semiring $\mathbb{N}_0$ satisfies this condition as well (Theorem \ref{Nnucleus}).
	
Finally, in \S\ref{sec:sclnotdistributiveidealmultiplication}, we explore the distributivity of the subtractive closure over ideal multiplication. Although Example \ref{subtractiveoperatornodistributemultiplication} (due to Dulin and Mosher \cite[Example~8]{DulinMosher1972}) demonstrates that this property does not hold in general, we show that entire austere hemirings do satisfy it (Proposition \ref{austereandentiremultiplicativedistributivity}). Furthermore, we prove that the interval semiring $T = [0,1]$ (Theorem \ref{NasehpourTsemiringdistribute}), the standard semiring $\mathbb{N}_0$ (Theorem \ref{Ndistributesmultiplicationthm}), and the max-plus algebra (Theorem \ref{maxplusalgebraideals}) all enjoy this distributive property.

\section{Closure operators and the subtractive closure}\label{sec:subtractiveclosure}

We begin by recalling the definition of a closure operator. Let $(P,\le)$ be a partially ordered set. A mapping $c \colon P \to P$ is called a closure operator on $P$ (see \S7.1 of \cite{DaveyPriestley2002}) if, for all $x,y \in P$, the following conditions hold:

\begin{description}
	\item[(CL1)] Extensivity: $x \le c(x)$;
	\item[(CL2)] Monotonicity: if $x \le y$, then $c(x) \le c(y)$;
	\item[(CL3)] Idempotence: $c(c(x)) = c(x)$.
\end{description}

\begin{proposition}\label{subsuperdistributivepro}
Let $c$ be a closure operator on a partially ordered set $(P,\le)$. Then the following statements hold:

\begin{enumerate}
	\item If $(P,\wedge)$ is a meet-semilattice, then $c$ is sub-distributive on $(P,\wedge)$; that is, $c(x \wedge y) \le c(x) \wedge c(y)$, for all $x,y \in P$.
	\item If $(P,\vee)$ is a join-semilattice, then $c$ is super-distributive on $(P, \vee)$; that is, $c(x) \vee c(y) \le c(x \vee y)$, for all $x,y \in P$.
\end{enumerate}

\end{proposition}

\begin{proof}
To prove the first statement, let $x, y \in P$. By the definition of the meet operation, $x \wedge y \le x$ and $x \wedge y \le y$. Applying the monotonicity of the closure operator (CL2), it follows that $c(x \wedge y) \le c(x)$ and $c(x \wedge y) \le c(y)$. By the universal property of the meet, this yields: \[c(x \wedge y) \le c(x) \wedge c(y).\] The second statement regarding join-semilattices follows immediately by order duality and is therefore omitted.
\end{proof}

\begin{proposition}\label{intersectionsubtractiveideals}
	An arbitrary intersection of subtractive ideals is subtractive.
\end{proposition}

\begin{proof}
Straightforward.
\end{proof}

For a hemiring $H$, we let $\Id_s(H)$ denote the set of all subtractive ideals of $H$. In view of Proposition \ref{intersectionsubtractiveideals}, we give the following definition.

\begin{definition}\label{subtractiveclosureidealdef}
	Let $X$ be a subset of a hemiring $H$. The subtractive ideal 
	\[ \scl_H(X) = \bigcap_{X \subseteq I \in \Id_s(H)} I \] 
	is called the subtractive closure of $X$.
\end{definition}
	
\begin{theorem}\label{subtractiveclosureproperties}
	Let $H$ be a hemiring and $\mathcal{P}(H)$ denote its power set. Then $\scl_H$ is a closure operator on $\mathcal{P}(H)$. Moreover, the following hold for all ideals $I$ and $J$ of $H$:
	\begin{enumerate}[\rm (a)]
		\item $\scl_H(I)=\{h \in H : h+i \in I \text{ for some } i \in I\}$.
		\item $\scl_H$ is super-distributive on $(\Id(H),+)$, i.e., $\scl_H(I)+\scl_H(J)\subseteq \scl_H(I+J)$.
		\item $\scl_H$ is sub-distributive on $(\Id(H),\cap)$, i.e., $\scl_H(I \cap J)\subseteq \scl_H(I)\cap \scl_H(J)$.
		\item $\scl_H$ is super-distributive on $(\Id(H),\cdot)$, i.e., $\scl_H(I) \scl_H(J) \subseteq \scl_H(IJ)$.
	\end{enumerate}
\end{theorem}

\begin{proof}
	That $\scl_H$ is a closure operator follows immediately from its definition and Proposition \ref{intersectionsubtractiveideals}.
	
	\smallskip
	\noindent (a) Let $A = \{h \in H : h+i \in I \text{ for some } i \in I\}$. It is straightforward to verify that $A$ is an ideal of $H$ containing $I$. To show $A$ is subtractive, let $a+b, a \in A$. Then $(a+b)+i_1 \in I$ and $a+i_2 \in I$ for some $i_1, i_2 \in I$. Since $I$ is an ideal, it contains $b + (a+i_2) + i_1 = (a+b) + i_1 + i_2$. Because $(a+i_2)+i_1 \in I$, it follows that $b \in A$. 
	
	Conversely, if $K$ is any subtractive ideal containing $I$, then for any $h \in A$, we have $h+i \in I \subseteq K$ for some $i \in I \subseteq K$. Subtractivity of $K$ implies $h \in K$, yielding $A \subseteq K$. Hence, $A = \scl_H(I)$.
	
	\smallskip
	\noindent (b) In light of Proposition \ref{subsuperdistributivepro}, this statement holds because $(\Id(H),+)$ is a join-semilattice.
	
	\smallskip
	\noindent (c) Similarly, this holds because $(\Id(H),\cap)$ is a meet-semilattice.
	
	\smallskip
	\noindent (d) Let $a \in \scl_H(I)$ and $b \in \scl_H(J)$. By (a), there exist $i_1, i_2 \in I$ and $j_1, j_2 \in J$ such that $a + i_1 = i_2$ and $b + j_1 = j_2$. Multiplying gives
	\[
	i_2 j_2 = ab + aj_1 + bi_1 + i_1 j_1 \in IJ \subseteq \scl_H(IJ).
	\]
	Since $i_2 j_1 \in IJ$ and $i_2 j_1 = (a + i_1)j_1 = aj_1 + i_1 j_1 \in IJ$, subtractivity of $\scl_H(IJ)$ yields $aj_1 \in \scl_H(IJ)$. Symmetrically, $bi_1 \in \scl_H(IJ)$. Now because $aj_1 + bi_1 + i_1 j_1 \in \scl_H(IJ)$ and $i_2 j_2 \in \scl_H(IJ)$, subtractivity again gives $ab \in \scl_H(IJ)$. Hence, $\scl_H(I) \scl_H(J) \subseteq \scl_H(IJ)$.
\end{proof}

\begin{remark}
	Theorem~\ref{subtractiveclosureproperties} establishes half of the full distributivity condition $c(a * b) = c(a) * c(b)$ for ideal operations under $\scl_H$: it is sub-distributive over intersection, and super-distributive over addition and multiplication. The remainder of this paper investigates conditions under which the respective reverse inclusions hold to achieve full distributivity.
\end{remark}

\section{Kuratowski semirings}\label{sec:kuratowski}

\begin{definition}\label{kuratowskimapdef}
	Let $(P, \le)$ be a join-semilattice with a least element $0$. We call a function $k \colon P \to P$ a Kuratowski map \cite[p. 180]{Sikorski1964} if the following conditions hold for all $x, y \in P$:
	\begin{description}
		\item[(CL1)] $x \le k(x)$,
		\item[(CL3)] $k(k(x)) = k(x)$,
		\item[(K1)] $k(0) = 0$,
		\item[(K2)] $k(x \vee y) = k(x) \vee k(y)$.
	\end{description}
\end{definition}

\begin{remark}\label{Kuratowskiimpliesmonotonicity}
	Every Kuratowski map is a closure operator. Indeed, if $x \le y$, then $x \vee y = y$. Applying \rm{(K2)} yields $k(x) \le k(x) \vee k(y) = k(x \vee y) = k(y)$, which establishes the required monotonicity \rm{(CL2)}. Furthermore, if $P$ has a greatest element $1$, then $k(1)=1$ follows trivially from \rm{(CL1)}. Finally, conditions \rm{(K1)} and \rm{(K2)} ensure that $k$ behaves precisely as a finite morphological dilation on $P$ \cite{BanonBarrera1993}.
\end{remark}
	
\begin{example}\label{NnotKuratowski}
	The subtractive closure $\scl_{\mathbb{N}_0}$ on the semiring $\mathbb{N}_0$ is not a Kuratowski map. Consider the subtractive ideals $I=2\mathbb{N}_0$ and $J=3\mathbb{N}_0$. Their sum $I+J = \mathbb{N}_0\setminus\{1\}$ is not subtractive (since $1+2 \in I+J$ and $2 \in I+J$, yet $1 \notin I+J$). Because $\scl_{\mathbb{N}_0}(I+J) = \mathbb{N}_0$, we obtain the strict inclusion:
	\[ \scl_{\mathbb{N}_0}(I) + \scl_{\mathbb{N}_0}(J) = I + J \subset \scl_{\mathbb{N}_0}(I+J). \]
\end{example}

In light of Theorem~\ref{subtractiveclosureproperties} and Example~\ref{NnotKuratowski}, we introduce the following definition.

\begin{definition}\label{def:kuratowski_hemiring}
	We call a hemiring $H$ a Kuratowski hemiring if the subtractive closure operator $\scl_H$ is a Kuratowski map on the lattice of ideals of $H$.
\end{definition}

\begin{proposition}\label{austereisKuratowski}
	Every subtractive hemiring and every austere hemiring is Kuratowski.
\end{proposition}

\begin{proof}
	If $H$ is a subtractive hemiring, then $\scl_H$ is the identity map, making it trivially Kuratowski. Now, let $I$ and $J$ be ideals of an austere hemiring $H$. If $I = J = \{0\}$, the result holds trivially. Otherwise, if at least one of $I$ or $J$ is nonzero, then $I+J \neq \{0\}$. Because $H$ is austere, the subtractive closure of any nonzero ideal is $H$ itself. This immediately yields $\scl_H(I+J) = H$ and $\scl_H(I) + \scl_H(J) = H$, establishing the equality.
\end{proof}

\begin{examples} 
	We provide examples of subtractive and austere semirings:
	\begin{enumerate}
		\item Any bounded distributive lattice $(L,\vee,\wedge,0,1)$ forms a subtractive semiring.
		\item For each $k \in \mathbb{N}$, the indigenous semiring $S_k$ \cite[Definition 4.3]{BehzadipourKoppelaarNasehpour2025} is austere \cite[Theorem 4.4(5)]{BehzadipourKoppelaarNasehpour2025}.
		\item For any information algebra $A$ and join-semilattice $L$, the overflow semiring $S = A \uplus_{\ord} L$ is austere \cite[Proposition 3.1]{Nasehpour2026}.
	\end{enumerate} 
\end{examples}

Recall that a hemiring $H$ is uniserial if its ideals are linearly ordered by inclusion \cite[Definition 2.2]{BehzadipourNasehpour2023}. 

\begin{proposition}\label{uniserialisKuratowski}
	Every uniserial hemiring is Kuratowski.
\end{proposition}

\begin{proof}
	Let $H$ be a uniserial hemiring. For any ideals $I, J \in \Id(H)$, we may assume without loss of generality that $I \subseteq J$. It follows that $I + J = J$, which yields $\scl_H(I + J) = \scl_H(J)$. Furthermore, monotonicity ensures $\scl_H(I) \subseteq \scl_H(J)$, implying $\scl_H(I) + \scl_H(J) = \scl_H(J)$. Equating the two results gives $\scl_H(I + J) = \scl_H(I) + \scl_H(J)$.
\end{proof}

Motivated by the fact that subtractive and austere hemirings are Kuratowski, we introduce a term for hemirings at the opposite end of this spectrum.

\begin{definition}\label{def:porous_hemiring}
We call	a hemiring $H$ porous if it contains a non-subtractive ideal; equivalently, there exists an ideal $I$ of $H$ such that $\scl_H(I) \neq I$.
\end{definition}

\begin{theorem}\label{NasehpourTsemiring}
	Let $T = [0,1]$ be the semiring equipped with addition $a \oplus b = \max(a,b)$ and multiplication defined by $a \odot b = 0$ for all $a,b \in [0,1)$, with $1 \odot a = a \odot 1 = a$ for all $a \in T$ \cite[Propositions 20 and 21]{Nasehpour2016}. The following statements hold:
	\begin{enumerate}[\rm (1)]
		\item $\Id(T) = \{T\} \cup \{ I \subseteq [0, 1) : 0 \in I \}$.
		\item An ideal $I \in \Id(T)$ is subtractive if and only if $I = T$, $I = [0, c]$ for $c \in [0, 1)$, or $I = [0, c)$ for $c \in (0, 1]$.
		\item For any ideal $I \in \Id(T)$, its subtractive closure is given by:
		\[
		\scl_T(I) = 
		\begin{cases} 
			[0, \sup I], & \text{if } \sup I \in I, \\ 
			[0, \sup I), & \text{if } \sup I \notin I. 
		\end{cases}
		\]
	\end{enumerate}
\end{theorem}

\begin{proof}
	(1) If $1 \in I$, then $t \odot 1 = t \in I$ for all $t \in T$, so $I = T$. If $1 \notin I$, any subset $I \subseteq [0,1)$ containing $0$ is automatically closed under addition because $\max(a,b) \in \{a,b\} \subseteq I$. It absorbs multiplication from $T$ since $t \odot a = 0 \in I$ for all $t < 1$, and $1 \odot a = a \in I$.
	
\smallskip \noindent (2) An ideal $I$ is subtractive if $\max(a,b) \in I$ and $a \in I \implies b \in I$. If $b \le a$, then $\max(a,b) = a \in I$, forcing $b \in I$. Thus, $I$ must be a down-closed set. Conversely, any down-closed ideal trivially satisfies this implication. Since $T$ is totally ordered, its down-closed subsets containing $0$ are precisely the listed intervals.
	
\smallskip \noindent (3) By Theorem~\ref{subtractiveclosureproperties}(a),
	\[
	\scl_T(I) = \{t \in T : \max(t,i) \in I \text{ for some } i \in I\}.
	\]
	If $t \le i$ for some $i \in I$, then $\max(t,i) = i \in I$, so $t \in \scl_T(I)$. Conversely, if $\max(t,i) \in I$ for some $i \in I$, then $t \le \max(t,i)$, whence $t \le j$ for some $j \in I$. Therefore
	\[
	\scl_T(I) = \{t \in T : t \le i \text{ for some } i \in I\} = {\downarrow}I.
	\]
	Let $s = \sup I$.
	\begin{itemize}
	\item If $s \in I$, then for any $t \in [0,s]$, we have $t \le s \in I$, implying $[0,s] \subseteq \scl_T(I)$. Conversely, if $t \in \scl_T(I)$, then $t \le i$ for some $i \in I$. Since $s = \sup I$, we have $i \le s$, forcing $t \in [0,s]$. Thus, $\scl_T(I) = [0,\sup I]$.
	\item If $s \notin I$, then for any $t < s$, by the definition of the supremum there exists some $i \in I$ such that $t < i < s$, which yields $t \in \scl_T(I)$ and establishes $[0,s) \subseteq \scl_T(I)$. Conversely, if $t \in \scl_T(I)$, then $t \le i$ for some $i \in I$. Because $i \le s$ and $s \notin I$, it follows that $i < s$, forcing $t < s$ and showing $\scl_T(I) \subseteq [0,s)$. Thus, $\scl_T(I) = [0,\sup I)$.
	\end{itemize} This completes the proof.
\end{proof}

\begin{lemma}\label{lem:sum_is_union}
	For any ideals $I, J \in \Id(T)$ of the semiring $T = [0,1]$ in Theorem \ref{NasehpourTsemiring}, their sum is equal to their union, i.e., $I + J = I \cup J$.
\end{lemma}

\begin{proof}
	By definition, $I + J = \{\max(i, j) : i \in I, j \in J\}$. Since $0 \in I \cap J$, setting $j=0$ yields $I \subseteq I+J$ and setting $i=0$ yields $J \subseteq I+J$; hence, $I \cup J \subseteq I + J$. Conversely, since $T$ is totally ordered, $\max(i, j) \in \{i, j\} \subseteq I \cup J$ for all $i \in I$ and $j \in J$, which gives $I + J \subseteq I \cup J$.
\end{proof}

\begin{theorem}\label{aporousKuratowskisemiring}
	The semiring $T = [0,1]$ in Theorem \ref{NasehpourTsemiring} is porous and not austere; however, it is a Kuratowski semiring.
\end{theorem}

\begin{proof}
	Because every subset of $[0,1)$ containing $0$ is an ideal, whereas subtractive ideals are strictly confined to down-closed intervals, $T$ contains plenty of ideals that are not subtractive (for example, the ideal $I = \{0, \frac{1}{2}\}$ has $\scl_T(I) = [0, \frac{1}{2}] \neq I$). Thus, $T$ is porous. Additionally, $T$ possesses numerous non-trivial subtractive ideals, meaning it is not austere. 
	
	To show that $T$ is a Kuratowski semiring, we must verify that $\scl_T(I + J) = \scl_T(I) + \scl_T(J)$ for all $I, J \in \Id(T)$. Recall that $\scl_T(K) = \{t \in T : t \le k \text{ for some } k \in K\}$. By Lemma \ref{lem:sum_is_union}, we have $I + J = I \cup J$. Therefore,
	\begin{align*}
		\scl_T(I + J) &= \scl_T(I \cup J) \\
		&= \{t \in T : t \le k \text{ for some } k \in I \cup J\} \\
		&= \{t \in T : t \le i \text{ for some } i \in I\} \cup \{t \in T : t \le j \text{ for some } j \in J\} \\
		&= \scl_T(I) \cup \scl_T(J).
	\end{align*}
	Since $\scl_T(I)$ and $\scl_T(J)$ are themselves down-closed intervals, they are valid ideals of $T$. A second application of Lemma \ref{lem:sum_is_union} yields $\scl_T(I) \cup \scl_T(J) = \scl_T(I) + \scl_T(J)$, completing the proof.
\end{proof}

\section{Nucleus semirings}\label{sec:nucleus}

\begin{definition}\label{nucleusmapdef}
	Let $(P, \le)$ be a meet-semilattice. A function $\nu \colon P \to P$ is a nucleus map \cite[\S5.3]{PicadoPultr2012} if the following conditions hold for all $x, y \in P$:
	\begin{description}
		\item[(CL1)] $x \le \nu(x)$,
		\item[(CL3)] $\nu(\nu(x)) = \nu(x)$,
		\item[(N)] $\nu(x \wedge y) = \nu(x) \wedge \nu(y)$.
	\end{description}
\end{definition}

\begin{remark}\label{nucleusimpliesmonotonicity}
Let $\nu$ be a nucleus on a meet-semilattice $(P, \le)$. Then $\nu$ is a closure operator because \[x \le y \implies x = x \wedge y \implies \nu(x) = \nu(x) \wedge \nu(y) \implies \nu(x) \le \nu(y).\] On complete lattices, $\nu$ defines a finite erosion; we recall that preserving arbitrary meets makes it a classical morphological erosion \cite{BanonBarrera1993}. A natural question arises in light of Theorem~\ref{subtractiveclosureproperties}(c): does the subtractive closure operator always distribute over finite intersections? As we will demonstrate, the answer is negative. To streamline the computational verification of this property, we introduce the concept of a prenucleus map:
\end{remark}

\begin{definition}
	Let $(P,\le)$ be a meet-semilattice. A function $\pi \colon P \to P$ is called a prenucleus \cite[\S11.5]{PicadoPultr2012} if the following conditions hold for all $x,y \in P$:
	\begin{description}
		\item[(CL1)] $x \le \pi(x)$,
		\item[(CL2)] $x \le y \implies \pi(x) \le \pi(y)$,
		\item[(PN)] $\pi(x) \wedge y \le \pi(x \wedge y)$.
	\end{description}
\end{definition}

\begin{proposition}\label{closureprenucleusgivesnucleus}
	Every operator on a meet-semilattice $P$ that is both a closure and a prenucleus is necessarily a nucleus.
\end{proposition}

\begin{proof}
	Let $\pi \colon P \to P$ be a closure operator and a prenucleus. The inequality $\pi(a \wedge b) \le \pi(a) \wedge \pi(b)$ follows directly from monotonicity. For the reverse inequality, applying the prenucleus condition $\textup{(PN)}$ successively to each component yields:
	\[
	\pi(a) \wedge \pi(b) \le \pi(a \wedge \pi(b)) \le \pi(\pi(a \wedge b)) = \pi(a \wedge b),
	\]
	where the second inequality arises from applying $\textup{(PN)}$ to $a \wedge \pi(b)$ under the monotonic action of $\pi$, and the final equality holds by idempotence.
\end{proof}

\begin{remark}
	Proposition~\ref{closureprenucleusgivesnucleus} is highly useful for computational purposes. Since the subtractive closure $\scl_H$ of a hemiring $H$ is always a closure operator, proving that $\scl_H$ acts as a nucleus map on the lattice of ideals reduces to verifying the prenucleus condition
	\[
	\scl_H(I) \cap J \subseteq \scl_H(I \cap J)
	\]
	for all ideals $I, J \in \Id(H)$. This requires evaluating the subtractive closure of only two ideals, whereas verifying the nucleus condition directly via
	\[
	\scl_H(I) \cap \scl_H(J) \subseteq \scl_H(I \cap J)
	\]
	would require computing the subtractive closure of three distinct ideals.
\end{remark}

\begin{theorem}\label{nonnucleusthm}
	The subtractive closure operator $\scl_T$ on the lattice of ideals of the semiring $T = [0,1]$ from Theorem~\ref{NasehpourTsemiring} is not a nucleus map.
\end{theorem}

\begin{proof}
	By the preceding remark, it suffices to show that the prenucleus condition fails. Consider the ideals $I = \{0, 1/3\}$ and $J = \{0, 1/2\}$ in $\Id(T)$. We have $\scl_T(J) \cap I = [0, 1/2] \cap \{0, 1/3\} = \{0, 1/3\}$, whereas $\scl_T(J \cap I) = \scl_T(\{0\}) = \{0\}$. Since $\{0, 1/3\} \not\subseteq \{0\}$, the condition $\scl_T(J) \cap I \subseteq \scl_T(J \cap I)$ is explicitly violated.
\end{proof}

\begin{remark}
In a recent paper, Goswami and Dube \cite[Lemma 2.1(7)]{GoswamiDube2024} claimed that for any family of ideals $\{I_\lambda\}_{\lambda \in \Lambda}$ of a semiring $R$,
	\[
	\mathcal{C}_k\left(\bigcap_{\lambda \in \Lambda} I_\lambda\right) = \bigcap_{\lambda \in \Lambda} \mathcal{C}_k(I_\lambda),
	\] where $\mathcal{C}_k$ denotes the $k$-closure operator, which coincides with our subtractive closure operator $\scl_R$. While the forward inclusion $\subseteq$ follows immediately from the monotonicity of closure operators, the reverse inclusion $\supseteq$ does not hold in general. Theorem~\ref{nonnucleusthm} provided an explicit counterexample to this claim. We therefore respectfully note that the equality stated in \cite[Lemma 2.1(7)]{GoswamiDube2024} is not valid for arbitrary semirings. To formally distinguish the class of hemirings for which this distributive equality actually holds, we introduce the following definition.
\end{remark}

\begin{definition}\label{nucleushemiringdef}
	We call a hemiring $H$ a nucleus hemiring if its subtractive closure operator $\scl_H$ acts as a nucleus map on the lattice of ideals $\Id(H)$.
\end{definition}

\begin{proposition}\label{subtractiveisnucleus}
Every subtractive hemiring is nucleus.
\end{proposition}

\begin{proof}
If $H$ is subtractive, $\scl_H$ is the identity operator, which is trivially a nucleus map.
\end{proof}

Recall that an $R$-module $M$ is called uniform if any two nonzero submodules of $M$ intersect nontrivially \cite{Lam1999}. A commutative ring $R$ is called uniform if $R$ as an $R$-module is uniform. This motivates us to give the following definition:

\begin{definition}
We call a hemiring $H$ uniform if any two nonzero ideals of $H$ intersect nontrivially. 
\end{definition}

\begin{remark}
	Every entire hemiring $E$ is uniform. This is because if $a \in I$ and $b \in J$ are nonzero elements of the ideals $I$ and $J$ of $E$, then $ab \in I \cap J$ is nonzero.
\end{remark}

\begin{proposition}\label{austereuniformisnucleus}
Every austere uniform hemiring is nucleus.
\end{proposition}

\begin{proof}
	 If $H$ is austere and the intersection of any two nonzero ideals is nonzero, the assertion follows by a dual argument to the proof of Proposition~\ref{austereisKuratowski}.
\end{proof}

\begin{proposition}\label{uniserialisnucleus}
A uniserial hemiring $H$ is nucleus.
\end{proposition}

\begin{proof}
	In a uniserial hemiring the ideals are totally ordered. For any $I,J\in\Id(H)$ we may assume $I\subseteq J$. Then $I\cap J=I$, so $\scl_H(I\cap J)=\scl_H(I)$. Monotonicity gives $\scl_H(I)\subseteq\scl_H(J)$, whence $\scl_H(I)\cap\scl_H(J)=\scl_H(I)$. Combining the two equalities yields $\scl_H(I\cap J)=\scl_H(I)\cap\scl_H(J)$.
\end{proof}

\begin{proposition}\label{plusunitelements}
	Let $I$ be an ideal of a semiring $S$ satisfying 
	\begin{equation}\label{plusunitelementseq}
		\exists s \in I \text{ and } \exists u \in U(S) \text{ such that } s + u \in I.
	\end{equation}
	Then $\scl_S(I)=S$. In particular, if $S$ is a uniform semiring and every nonzero ideal of $S$ satisfies \eqref{plusunitelementseq}, then $S$ is a nucleus semiring.
\end{proposition}

\begin{proof}
	First, we show that $\scl_S(I) = S$. By the definition of the subtractive closure, since $s \in I$ and $s+u \in I$, subtractivity yields $u \in \scl_S(I)$. Because $u$ is a unit and is contained in the ideal $\scl_S(I)$, the ideal must be the whole semiring, so $\scl_S(I) = S$. To prove the ``in particular'' part, recall that $\scl_S$ is a nucleus if it preserves finite intersections: $\scl_S(I \cap J) = \scl_S(I) \cap \scl_S(J)$. If either $I$ or $J$ is the zero ideal, this holds trivially. If both $I$ and $J$ are nonzero, their intersection $I \cap J$ is also nonzero because $S$ is a uniform semiring. Thus $I$, $J$, and $I \cap J$ all satisfy \eqref{plusunitelementseq}. Applying the first part of the proof to each yields $\scl_S(I) = \scl_S(J) = \scl_S(I \cap J) = S$. Consequently, $\scl_S(I \cap J) = \scl_S(I) \cap \scl_S(J) = S$, and thus $\scl_S$ acts as a nucleus map on $\Id(S)$.
\end{proof}

\begin{lemma}\label{gcdintersection}
	Let $I$ and $J$ be nonzero ideals of the standard semiring $\mathbb{N}_0$. Then
	\[
	\gcd(I\cap J) = \operatorname{lcm}(\gcd(I), \gcd(J)).
	\]
\end{lemma}

\begin{proof}
Let $d_1 = \gcd(I)$, $d_2 = \gcd(J)$, $\ell = \operatorname{lcm}(d_1, d_2)$, and $g = \gcd(I \cap J)$. Since $I \subseteq d_1\mathbb{N}_0$ and $J \subseteq d_2\mathbb{N}_0$, their intersection satisfies $I \cap J \subseteq \ell\mathbb{N}_0$. Thus, $\ell$ divides every element of $I \cap J$, which implies $\ell \mid g$.
	
To prove the reverse divisibility, note that $(1/d_1)I$ and $(1/d_2)J$ are ideals of $\mathbb{N}_0$ with greatest common divisors equal to $1$. By Lemma 2.1 of \cite{RosalesSanchez2009}, $(1/d_1)I$ and $(1/d_2)J$ are numerical semigroups and are therefore cofinite in $\mathbb{N}_0$. Consequently, $I$ and $J$ contain all sufficiently large multiples of $d_1$ and $d_2$, respectively. It follows that $I \cap J$ contains all sufficiently large multiples of $\ell$. Choosing an integer $m$ large enough such that $m\ell, (m+1)\ell \in I \cap J$, we deduce that $g$ must divide both elements since $g = \gcd(I \cap J)$. Hence, $g$ divides their difference: \[g \mid ((m+1)\ell - m\ell) = \ell.\] Combining $\ell \mid g$ and $g \mid \ell$ yields $g = \ell$, completing the proof.
\end{proof}

\begin{lemma}\label{closureofnonzeroideal}
	Let $I$ be a nonzero ideal of $\mathbb{N}_0$ with $\gcd(I) = d$. Then $\scl_{\mathbb{N}_0}(I) = d\mathbb{N}_0$.
\end{lemma}

\begin{proof}
	Set $I' = \frac{1}{d}I \subseteq \mathbb{N}_0$. Then $\gcd(I') = 1$. Since $I \neq \{0\}$, $I'$ is a numerical semigroup, hence by Lemma 2.1 of \cite{RosalesSanchez2009}, it is cofinite and contains two consecutive integers $n, n+1$. Consequently, $dn, d(n+1) \in I$. Since $d(n+1) = dn + d$, we have $dn, dn + d \in I \subseteq \scl_{\mathbb{N}_0}(I)$. Because $\scl_{\mathbb{N}_0}(I)$ is subtractive, $d \in \scl_{\mathbb{N}_0}(I)$. As $\scl_{\mathbb{N}_0}(I)$ is an ideal and $d \in \scl_{\mathbb{N}_0}(I)$, it follows that $d\mathbb{N}_0 \subseteq \scl_{\mathbb{N}_0}(I)$. The reverse inclusion $\scl_{\mathbb{N}_0}(I) \subseteq d\mathbb{N}_0$ holds since $I \subseteq d\mathbb{N}_0$ and $d\mathbb{N}_0$ is subtractive. Thus $\scl_{\mathbb{N}_0}(I) = d\mathbb{N}_0$.
\end{proof}

\begin{theorem}\label{Nnucleus}
The standard semiring $\mathbb{N}_0$ is neither austere nor Kuratowski; however, it is nucleus.
\end{theorem}

\begin{proof}
The standard semiring $\mathbb{N}_0$ is not austere because $k\mathbb{N}_0$ is a proper nonzero subtractive ideal for any $k > 1$. By Example~\ref{NnotKuratowski}, $\mathbb{N}_0$ is not Kuratowski. To show that $\mathbb{N}_0$ is nucleus, let $I$ and $J$ be nonzero ideals of $\mathbb{N}_0$. Applying Lemma~\ref{gcdintersection} and Lemma~\ref{closureofnonzeroideal}, we obtain:
	\begin{align*}
		\scl_{\mathbb{N}_0}(I\cap J)
		&= \gcd(I\cap J)\mathbb{N}_0 = \operatorname{lcm}\!\bigl(\gcd(I),\gcd(J)\bigr)\mathbb{N}_0 \\
		&= \gcd(I)\mathbb{N}_0 \cap \gcd(J)\mathbb{N}_0 = \scl_{\mathbb{N}_0}(I)\cap \scl_{\mathbb{N}_0}(J).
	\end{align*}
	If either ideal is zero, the equality holds trivially since $\scl_{\mathbb{N}_0}(\{0\}) = \{0\}$. Hence, $\scl_{\mathbb{N}_0}$ preserves finite intersections and is therefore a nucleus map.
\end{proof}

\section{Multiplicative distributivity of the subtractive closure}\label{sec:sclnotdistributiveidealmultiplication}

In Theorem~\ref{subtractiveclosureproperties}, we proved that for any ideals $I$ and $J$ of a hemiring $H$,
\[
\scl_H(I)\scl_H(J)\subseteq \scl_H(IJ).
\]
The following example shows that the reverse inclusion need not hold. The construction is due to Dulin and Mosher \cite[Example~8]{DulinMosher1972}.

\begin{example}[Dulin-Mosher hemiring]\label{subtractiveoperatornodistributemultiplication}
	Let $H=\mathbb{N}_0$ be the set of nonnegative integers. Define binary operations $\oplus$ and $\odot$ on $H$ by
	\[
	a \oplus b =
	\begin{cases}
		\max(a,b), & \text{if } a\le 6 \text{ or } b\le 6,\\
		a+b, & \text{otherwise},
	\end{cases}
	\]
	and
	\[
	a \odot b =
	\begin{cases}
		\min(a,b), & \text{if } a\le 6 \text{ or } b\le 6,\\
		ab, & \text{otherwise},
	\end{cases}
	\]
	where $a+b$ and $ab$ denote the usual addition and multiplication of integers. Let
	\[
	I=\{0,1,2,3,4,5,6\}\cup\{2k\in\mathbb{N}_0:2k>6\}
	\]
	and
	\[
	J=\{0,1,2,3,4,5,6\}\cup\{3k\in\mathbb{N}_0:3k>6\}.
	\]
	Then $I$ and $J$ are subtractive ideals of $H$. Hence $\scl_H(I)=I$ and $\scl_H(J)=J$. Therefore
	\[
	\scl_H(I)\odot \scl_H(J)=I \odot J.
	\] 
	However, $I\odot J$ is not subtractive. Indeed,
	\[
	96=8\odot12\in I\odot J,\qquad 120=10\odot12\in I\odot J,
	\]
	and $96\oplus24=120$, while $24\notin I\odot J$. Consequently,
	\[
	\scl_H(I) \odot \scl_H(J)\subset \scl_H(I \odot J).
	\]
\end{example}

\begin{proposition}\label{DulinMoshernotKuratowski}
Let $H$, $I$, and $J$ be as in Example \ref{subtractiveoperatornodistributemultiplication}. Then
	\[
	I \oplus J=\mathbb N_0\setminus\{7,11,13\} \quad \text{and} \quad 
	\scl_H(I \oplus J)=\mathbb N_0.\] Therefore, the Dulin-Mosher hemiring is not Kuratowski.
\end{proposition}

\begin{proof}
Clearly $I \cup J \subseteq I \oplus J$. If $n \ge 17$ is odd, write $n = (n-9) \oplus 9$; because $n-9 \ge 8$ is even, hence it is in $I$ and since $9 \in J$, we obtain $n \in I \oplus J$. If $n \ge 17$ is even, $n = 2k$ with $k \ge 9$, so $n \in I \subseteq I \oplus J$. The even numbers in $\{7,\dots,16\}$ lie in $I$, and the multiples of $3$ in that interval lie in $J$.  
	For any $x \in I$ with $x > 6$ we have $x \ge 8$, and for any $y \in J$ with $y > 6$ we have $y \ge 9$; thus $x \oplus y = x+y \ge 17$. Consequently no pair from $I$ and $J$ can sum to $7, 11,$ or $13$. Therefore
	\[
	I \oplus J = (I \cup J) \cup \{n \in \mathbb N_0 : n \ge 17\} 
	= \mathbb N_0 \setminus \{7, 11, 13\}.
	\]
	Now $10, 17 \in I \oplus J$ and $10 \oplus 7 = 17$ forces $7 \in \scl_H(I \oplus J)$;  
	$8, 19 \in I \oplus J$ and $8 \oplus 11 = 19$ gives $11 \in \scl_H(I \oplus J)$;  
	$8, 21 \in I \oplus J$ and $8 \oplus 13 = 21$ gives $13 \in \scl_H(I \oplus J)$.  
	Hence $\scl_H(I \oplus J) = \mathbb N_0$, and we have
	\[
	\scl_H(I) \oplus \scl_H(J) = I \oplus J 
	= \mathbb N_0 \setminus \{7,11,13\} 
	\subset \mathbb N_0 = \scl_H(I \oplus J),
	\]
	so the Dulin--Mosher hemiring is not Kuratowski.
\end{proof}

\begin{theorem}\label{DMidealsclassification}
Let $H$ be the Dulin-Mosher hemiring. Let $I$ be an ideal of $H$.
	\begin{enumerate}
		\item If $I$ contains at least one element strictly greater than $6$, then $I$ contains $D_6 = \{0,1,\dots,6\}$, called a large ideal of $H$.
		\item If all elements of $I$ are $\leq 6$, then $ I = D_m = \{0,1,\dots,m\}$ for some $0 \le m \le 6$. In this case $I$ is subtractive, called a small ideal of $H$.
	\end{enumerate}
\end{theorem}

\begin{proof}
(1) Let $I$ be an ideal of $H$. If $I$ contains any element $m > 6$, then for every $k \le 6$, we have $k \odot m = \min(k,m) = k \in I$. This forces $D_6 \subseteq I$.
	
\smallskip \noindent (2) If $I$ contains no elements greater than $6$, it must possess a maximum element $m \le 6$. For any $k \le m$, we again have $k \odot m = \min(k,m) = k \in I$. Since $m$ is maximal, this explicitly forces $I = \{0,1,\dots,m\} = D_m$. To verify that every $D_m$ is subtractive, suppose $a \oplus b \in D_m$ and $a \in D_m$. Because $a \oplus b \le m \le 6$, the operation cannot occur in the standard addition tier ($>6$). Thus, the operation evaluates as $a \oplus b = \max(a,b) \le m$. This immediately dictates $b \le m$, proving $b \in D_m$.
\end{proof}

\begin{theorem}\label{DMlargeidealclosure}
	Let $H$ be the Dulin-Mosher hemiring and $I$ be a large ideal of $H$. Let $I_{>6} = \{x \in I : x > 6\}$ and define $d = \gcd(I_{>6})$. Then the subtractive closure of $I$ is given by
	\[
	\scl_H(I) = D_6 \cup \{ x \in \mathbb{N}_0 : x > 6 \text{ and } d \mid x \},
	\]
	where $D_6 = \{0,1,2,3,4,5,6\}$.
\end{theorem}

\begin{proof}
	Let $S = D_6 \cup \{ x \in \mathbb{N}_0 : x > 6 \text{ and } d \mid x \}$. We first verify that $S$ is a subtractive ideal of $H$. Clearly, $D_6 \subseteq S$. For any $a, b \in S$, if either is $\le 6$, their sum in $H$ is $\max(a,b)$, which remains in $S$. If both $a, b > 6$, then $a \oplus b = a + b$. Since both are multiples of $d$, their sum is a multiple of $d$ strictly greater than $6$, so $a \oplus b \in S$. For multiplication, let $h \in H$ and $a \in S$. If $h \le 6$ or $a \le 6$, then $h \odot a = \min(h, a) \le 6 \in D_6 \subseteq S$. If $h > 6$ and $a > 6$, then $a \in S$ implies $d \mid a$. Thus, $h \odot a = ha > 6$ is also a multiple of $d$, which yields $ha \in S$. Therefore, $S$ is closed under multiplication, establishing that $S$ is an ideal of $H$. 
	
	To see that $S$ is subtractive, suppose $a \oplus b \in S$ and $a \in S$. If $a \le 6$ or $b \le 6$, then $a \oplus b = \max(a,b)$. If $a \le 6$, then $\max(a,b) \in S$ forces $b \in S$. If $b \le 6$, then $b \in D_6 \subseteq S$. If both $a, b > 6$, then $a \oplus b = a+b \in S$. This implies $a+b = md$ and $a = nd$ for some integers $m > n$, forcing $b = (m-n)d$. Since $b > 6$, it follows that $b \in S$. Hence, $S$ is a subtractive ideal. By construction, $I_{>6}$ consists entirely of multiples of $d$, so $I \subseteq S$. The minimality of the subtractive closure dictates $\scl_H(I) \subseteq S$.
	
	Conversely, let $x \in S$. If $x \le 6$, then $x \in D_6 \subseteq I \subseteq \scl_H(I)$. If $x > 6$, then $x = kd$ for some integer $k$. By the definition of $d = \gcd(I_{>6})$, and because $I_{>6}$ is closed under addition and absorbs multiplication by elements $>6$, the scaled set $\frac{1}{d}I_{>6}$ is cofinite in $\mathbb{N}_0$ (analogous to the numerical semigroup property established in Lemma \ref{closureofnonzeroideal}). Thus, $I_{>6}$ contains all sufficiently large multiples of $d$. Choose an integer $N$ large enough such that $Nd \in I_{>6}$ and $(N+k)d \in I_{>6}$, with $Nd > 6$. Then in $H$ we have
	\[
	x \oplus Nd = kd + Nd = (k+N)d \in I.
	\]
	Because $Nd \in I$ and $x \oplus Nd \in I$, the definition of the subtractive closure directly yields $x \in \scl_H(I)$. This proves $S \subseteq \scl_H(I)$, establishing the desired equality.
\end{proof}

\begin{theorem}\label{DulinMosherIsNucleus}
	The Dulin-Mosher hemiring is a nucleus hemiring.
\end{theorem}

\begin{proof}
	Let $H$ be the Dulin-Mosher hemiring and $I, J$ be ideals of $H$. We must show that $\scl_H(I \cap J) = \scl_H(I) \cap \scl_H(J)$. By Theorem \ref{DMidealsclassification}, every ideal of $H$ is either a small ideal $D_m$ ($m \le 6$) or a large ideal containing $D_6$. We consider three exhaustive cases:
	
	\medskip\noindent\textbf{Case 1: Both ideals are small.} 
	Let $I = D_m$ and $J = D_n$. Because the small ideals are totally ordered by inclusion, $I \cap J = D_{\min(m,n)}$. Since all small ideals are inherently subtractive, their subtractive closures are themselves. Thus,
	\[ 
	\scl_H(D_m \cap D_n) = D_{\min(m,n)} = D_m \cap D_n = \scl_H(D_m) \cap \scl_H(D_n). 
	\]
	
	\medskip\noindent\textbf{Case 2: One ideal is small and one is large.} 
	Let $I = D_m$ for $m \le 6$, and $J$ be a large ideal. By Theorem \ref{DMidealsclassification}, $D_6 \subseteq J$, which strictly implies $D_m \subseteq J$. Therefore, $I \cap J = D_m$. Taking the subtractive closure, we have $\scl_H(I \cap J) = \scl_H(D_m) = D_m$. 
	
	Conversely, the subtractive closure $\scl_H(J)$ is a subtractive ideal containing $J$, meaning it also contains $D_6$. Consequently, $\scl_H(I) \cap \scl_H(J) = D_m \cap \scl_H(J) = D_m$. Thus, the equality holds.
	
	\medskip\noindent\textbf{Case 3: Both ideals are large.} 
	Let $I$ and $J$ be large ideals. By Theorem \ref{DMlargeidealclosure}, we can express their subtractive closures as $\scl_H(I) = D_6 \cup d_1 \mathbb{N}_{>6}$ and $\scl_H(J) = D_6 \cup d_2 \mathbb{N}_{>6}$, where $d_1 = \gcd(I_{>6})$, $d_2 = \gcd(J_{>6})$, and the notation $k \mathbb{N}_{>6}$ denotes the set $\{x \in \mathbb{N}_0 : x > 6 \text{ and } k \mid x\}$. 
	
	The intersection of their subtractive closures is:
	\begin{align*}
		\scl_H(I) \cap \scl_H(J) &= (D_6 \cup d_1 \mathbb{N}_{>6}) \cap (D_6 \cup d_2 \mathbb{N}_{>6}) \\
		&= D_6 \cup (d_1 \mathbb{N}_{>6} \cap d_2 \mathbb{N}_{>6}) \\
		&= D_6 \cup \operatorname{lcm}(d_1, d_2) \mathbb{N}_{>6}.
	\end{align*}
	
	Now consider the ideal $I \cap J$. Any element in $(I \cap J)_{>6}$ belongs to both $I_{>6}$ and $J_{>6}$, meaning it is a multiple of both $d_1$ and $d_2$. Thus, it is a multiple of $\ell = \operatorname{lcm}(d_1, d_2)$. This dictates that $\ell$ divides $\gcd((I \cap J)_{>6})$. 
	
	On the other hand, analogous to the proof of Lemma \ref{gcdintersection}, $I_{>6}$ and $J_{>6}$ contain all sufficiently large multiples of $d_1$ and $d_2$, respectively. Hence, their intersection contains all sufficiently large multiples of $\ell$. This forces the greatest common divisor to be exactly $\ell$, meaning $\gcd((I \cap J)_{>6}) = \ell$. 
	
	Applying Theorem \ref{DMlargeidealclosure} to the large ideal $I \cap J$ with its greatest common divisor $\ell$, we obtain:
	\[
	\scl_H(I \cap J) = D_6 \cup \ell \mathbb{N}_{>6} = \scl_H(I) \cap \scl_H(J).
	\]
	Since the equality holds in all possible combinations, $H$ acts as a nucleus hemiring.
\end{proof}

\begin{theorem}\label{NasehpourTsemiringdistribute}
	Let $T = [0,1]$ be the semiring introduced in Theorem~\ref{NasehpourTsemiring}. The subtractive closure operator $\scl_T$ distributes over ideal multiplication. That is, for any ideals $I, J \in \Id(T)$,
	$$ \scl_T(I \odot J) = \scl_T(I) \odot \scl_T(J). $$
\end{theorem}

\begin{proof}
	By Theorem~\ref{NasehpourTsemiring}(1), every ideal of $T$ is either the entire semiring $T$ or a subset of $[0, 1)$. We consider two exhaustive cases:
	
	\medskip\noindent\textbf{Case 1: Both $I$ and $J$ are proper ideals.} 
	If $I \neq T$ and $J \neq T$, then $I \subseteq [0,1)$ and $J \subseteq [0,1)$. By the definition of multiplication in $T$, $a \odot b = 0$ for all $a \in I$ and $b \in J$. Thus, the ideal product collapses to $I \odot J = \{0\}$, and its subtractive closure is $\scl_T(I \odot J) = \{0\}$. 
	
	Now consider their individual closures. By Theorem~\ref{NasehpourTsemiring}(3), the subtractive closure of any ideal in $[0,1)$ is bounded by its supremum. Crucially, even if the supremum is $1$, the closure remains open at $1$ because $1 \notin I$. Therefore, $1 \notin \scl_T(I)$ and $1 \notin \scl_T(J)$, meaning both closures are strictly contained in $[0, 1)$. Consequently, multiplying them yields:
	$$ \scl_T(I) \odot \scl_T(J) = \{0\} = \scl_T(I \odot J). $$
	
	\medskip\noindent\textbf{Case 2: At least one ideal is $T$.} 
	Without loss of generality, let $I = T$, and let $J$ be any ideal in $\Id(T)$. Because $T$ is a semiring with a multiplicative identity ($1 \in T$), it universally holds that $T \odot J = J$. Taking the subtractive closure of both sides gives $\scl_T(T \odot J) = \scl_T(J)$. Furthermore, since $\scl_T(J)$ is itself an ideal of $T$, multiplying it by $T$ leaves it unchanged, yielding $T \odot \scl_T(J) = \scl_T(J)$. Because $\scl_T(T) = T$, we can substitute to find:
	\[ \scl_T(T \odot J) = \scl_T(J) = \scl_T(T) \odot \scl_T(J). \] Since equality holds in both cases, $\scl_T$ distributes over ideal multiplication.
\end{proof}

\begin{proposition}\label{austereandentiremultiplicativedistributivity} 
	Let $H$ be a hemiring. Then $\scl_H(I)\scl_H(J) = \scl_H(IJ)$ for all $I,J\in\Id(H)$ if either
	\begin{enumerate}[\rm (a)]
		\item $H$ is subtractive, or
		\item $H$ is austere, entire, and $H^2=H$.
	\end{enumerate}
\end{proposition}

\begin{proof}
	(a) Trivial, since $\scl_H$ is the identity map.
	
	(b) If $I=\{0\}$ or $J=\{0\}$, both sides are $\{0\}$. Otherwise $I,J\neq\{0\}$ implies $IJ\neq\{0\}$ by entirety, so austerity gives $\scl_H(I)=\scl_H(J)=\scl_H(IJ)=H$. Hence $\scl_H(I)\scl_H(J)=H\cdot H=H^2=H=\scl_H(IJ)$.
\end{proof}

\begin{theorem}\label{Ndistributesmultiplicationthm}
	The subtractive closure operator $\scl_{\mathbb{N}_0}$ on the standard semiring $\mathbb{N}_0$ distributes over the multiplication of ideals, i.e., for any ideals $I$ and $J$ of $\mathbb{N}_0$,
	\[
	\scl_{\mathbb{N}_0}(I)\scl_{\mathbb{N}_0}(J) = \scl_{\mathbb{N}_0}(IJ).
	\]
\end{theorem}

\begin{proof}
	If $I = \{0\}$ or $J = \{0\}$, then $IJ = \{0\}$, and both sides trivially equal $\{0\}$. Now assume $I \neq \{0\}$ and $J \neq \{0\}$. By Lemma~\ref{closureofnonzeroideal}, the subtractive closure of any nonzero ideal $A$ of $\mathbb{N}_0$ is given by $\scl_{\mathbb{N}_0}(A) = \gcd(A)\mathbb{N}_0$. Let $d_1 = \gcd(I)$ and $d_2 = \gcd(J)$. Then we have:
	\[
	\scl_{\mathbb{N}_0}(I)\scl_{\mathbb{N}_0}(J) = (d_1 \mathbb{N}_0)(d_2 \mathbb{N}_0) = d_1 d_2 \mathbb{N}_0.
	\]
	To compute $\scl_{\mathbb{N}_0}(IJ)$, we determine $\gcd(IJ)$. Since every element of $IJ$ is a finite sum of products of the form $i \cdot j$ with $i \in I$ and $j \in J$, and since $d_1 \mid i$ and $d_2 \mid j$, it follows that $d_1 d_2 \mid ij$. Thus, $d_1 d_2$ is a common divisor of all $ij \in IJ$, which implies $d_1 d_2 \mid \gcd(IJ)$. Conversely, by B\'{e}zout's identity over $\mathbb{Z}$, there exist finite sets of elements $\{i_1, \dots, i_n\} \subseteq I$ and $\{j_1, \dots, j_m\} \subseteq J$ along with integers $c_1, \dots, c_n, b_1, \dots, b_m \in \mathbb{Z}$ such that:
	\[
	\sum_{k=1}^n c_k i_k = d_1 \quad \text{and} \quad \sum_{s=1}^m b_s j_s = d_2.
	\]
	Multiplying these two equations yields:
	\[
	d_1 d_2 = \sum_{k=1}^n \sum_{s=1}^m (c_k b_s) (i_k j_s).
	\]
	Since each product $i_k j_s \in IJ$, the greatest common divisor $\gcd(IJ)$ must divide every such product, and consequently, it must divide their integer linear combination $d_1 d_2$. Since both $d_1 d_2$ and $\gcd(IJ)$ are positive integers, we conclude that $\gcd(IJ) = d_1 d_2$. Thus,
	\[
	\scl_{\mathbb{N}_0}(IJ) = \gcd(IJ)\mathbb{N}_0 = d_1 d_2 \mathbb{N}_0 = \scl_{\mathbb{N}_0}(I)\scl_{\mathbb{N}_0}(J),
	\]
	which completes the proof.
\end{proof}

\begin{theorem}\label{maxplusalgebraideals}
Consider the max-plus algebra over $S = \mathbb{N}_0 \cup \{-\infty\}$ equipped with addition $x \oplus y = \max(x,y)$ and multiplication $x \odot y = x + y$. Then the following statements hold:

\begin{enumerate}
	
	\item The ideals of $S$ are exactly
	\[
	I_n = \{k \in \mathbb{N}_0 \mid k \ge n\} \cup \{-\infty\}, \quad n \in \mathbb{N}_0,
	\]
	together with the zero ideal $I_{-\infty} = \{-\infty\}$. These ideals form a chain:
	\[
	I_{-\infty} \subset \cdots \subset I_3 \subset I_2 \subset I_1 \subset I_0 = S.
	\] In particular, $S$ is both Kuratowski and nucleus.
	
	\item $S$ is an austere information algebra. In particular, the subtractive closure $\scl_S$ distributes over the ideal multiplication. 
\end{enumerate}

\end{theorem}

\begin{proof}
(1) Every ideal of $S$ must contain the additive identity $-\infty$. Let $I \neq \{-\infty\}$ be a nonzero ideal and set $n = \min(I \cap \mathbb{N}_0)$. For any $m \in \mathbb{N}_0$ with $m \ge n$, we have $m - n \in \mathbb{N}_0 \subset S$, whence
		\[
		m = (m-n) + n = (m-n) \odot n \in I.
		\] Since addition in $S$ is the maximum operation, taking sums of elements in $I$ generates no element smaller than $n$. Thus $I = \{k \in \mathbb{N}_0 : k \ge n\} \cup \{-\infty\} = I_n$. Conversely, each $I_n$ is clearly closed under $\max$ and under multiplication by elements of $S$, so every $I_n$ is an ideal.
		The chain structure $I_m \subset I_n \iff n < m$ is immediate. 
		Therefore $S$ is uniserial, and by Propositions~\ref{uniserialisKuratowski} and~\ref{uniserialisnucleus}, every uniserial hemiring is both Kuratowski and nucleus.
		
\smallskip \noindent (2) By (1), the ideals of $S$ are $I_{-\infty} = \{-\infty\}$, $I_0 = S$, and $I_n = \{k \in \mathbb{N}_0 : k \ge n\} \cup \{-\infty\}$ for $n \ge 1$. 
		For $n \ge 1$, take $x = n \in I_n$ and $y = n-1 \notin I_n$. 
		Then $x \oplus y = \max(n, n-1) = n \in I_n$, yet $y \notin I_n$. 
		Hence $I_n$ is not subtractive, and the only subtractive ideals of $S$ are the trivial ones $\{-\infty\}$ and $S$. 
		Thus $S$ is an austere semiring. 
		It is clear that $S$ is zerosumfree, since $\max(a,b) = -\infty$ implies $a = b = -\infty$, and entire, since $a + b = -\infty$ implies $a = -\infty$ or $b = -\infty$. 
		Hence $S$ is an information algebra. 
		By Proposition~\ref{austereandentiremultiplicativedistributivity}, the subtractive closure $\scl_S$ distributes over the multiplication of ideals, completing the proof.
\end{proof}

\appendix

\section{Yet another family of counterexamples}\label{appendix:polynomial}

We now construct a broader class of counterexamples to demonstrate that the nucleus property fails for a wide variety of semirings. Specifically, we will prove that the polynomial semiring in three variables over any zerosumfree semiring fails to be a nucleus semiring. Before establishing this result, we develop the necessary preparatory framework concerning graded semirings and the algebraic behavior of homogeneous ideals.

\begin{definition}
	Let $(\Gamma,+)$ be a commutative cancellative monoid. A semiring $S$ is called graded by $\Gamma$ (or $\Gamma$-graded) if there exists a family $\{S_\gamma\}_{\gamma \in \Gamma}$ of additive submonoids of $S$ such that:
	\begin{enumerate}
		\item $S = \bigoplus_{\gamma \in \Gamma} S_\gamma$ (direct sum as additive monoids), meaning every $x \in S$ can be written uniquely as $x = \sum_{\gamma \in \Gamma} x_\gamma$ with $x_\gamma \in S_\gamma$ and only finitely many $x_\gamma \neq 0$;
		\item $S_\alpha \cdot S_\beta \subseteq S_{\alpha+\beta}$ for all $\alpha, \beta \in \Gamma$.
	\end{enumerate}
	For $x \in S$, the elements $x_\gamma \in S_\gamma$ in the unique decomposition $x = \sum_{\gamma \in \Gamma} x_\gamma$ are called the homogeneous components of $x$. An element $x$ is called homogeneous if $x \in S_\gamma$ for some $\gamma \in \Gamma$.
\end{definition}

\begin{definition}
	Let $S = \bigoplus_{\gamma \in \Gamma} S_\gamma$ be a commutative semiring graded by a commutative cancellative monoid $\Gamma$. An ideal $I$ of $S$ is called homogeneous if $I$ is generated by homogeneous elements.
\end{definition}

\begin{theorem}\label{homogeneouscomponents}
	Let $S = \bigoplus_{\gamma \in \Gamma} S_\gamma$ be a commutative semiring graded by a commutative cancellative monoid $\Gamma$. Let $I$ be an ideal of $S$. Then the following are equivalent:
	
	\begin{enumerate}[(a)]
		\item $I = \bigoplus_{\gamma \in \Gamma} (I \cap S_\gamma)$,
		\item for every $x \in I$, all homogeneous components $x_\gamma$ of $x$ belong to $I$,
		\item $I$ is homogeneous.
	\end{enumerate}
\end{theorem}

\begin{proof}
	We prove $(a) \Rightarrow (b) \Rightarrow (c) \Rightarrow (a)$.
	
	\medskip
	\noindent\textbf{$(a) \Rightarrow (b)$:}
	Assume $I = \bigoplus_{\gamma \in \Gamma} (I \cap S_\gamma)$. Let $x \in I$. Then by assumption, $x \in \bigoplus_{\gamma \in \Gamma} (I \cap S_\gamma)$, so $x = \sum_{\gamma \in \Gamma} y_\gamma$ with $y_\gamma \in I \cap S_\gamma$. But the decomposition of $x$ into homogeneous components is unique (since the sum is direct), so $x_\gamma = y_\gamma \in I \cap S_\gamma \subseteq I$ for each $\gamma$. Hence all homogeneous components of $x$ lie in $I$.
	
	\medskip
	\noindent\textbf{$(b) \Rightarrow (c)$:}
	Assume that for every $x \in I$, all homogeneous components $x_\gamma$ belong to $I$. Let $H = \{x_\gamma : x \in I, \gamma \in \Gamma\}$, the set of all homogeneous components of elements of $I$. Then $H \subseteq I$ by assumption, and $H$ consists of homogeneous elements. Every $x \in I$ can be written as $x = \sum_{\gamma} x_\gamma$, a sum of elements of $H$. Thus $I$ is generated by the homogeneous elements in $H$, so $I$ is generated by homogeneous elements.
	
	\medskip
	\noindent\textbf{$(c) \Rightarrow (a)$:}
	Assume $I$ is generated by a set $H$ of homogeneous elements. We must show $I = \bigoplus_{\gamma \in \Gamma} (I \cap S_\gamma)$. The inclusion \[\bigoplus_{\gamma \in \Gamma} (I \cap S_\gamma) \subseteq I\] is clear, since each $I \cap S_\gamma \subseteq I$. For the reverse inclusion $I \subseteq \bigoplus_{\gamma \in \Gamma} (I \cap S_\gamma)$, let $x \in I$. Since $S$ is commutative and $I$ is generated by $H$, we can write
	\[
	x = \sum_{k} a_k h_k
	\]
	for some finite collection of $h_k \in H$ and $a_k \in S$. Each $h_k$ is homogeneous, say $h_k \in S_{\delta_k}$. Write $a_k = \sum_{\alpha} a_{k,\alpha}$ as the sum of its homogeneous components with $a_{k,\alpha} \in S_\alpha$. Then
	\[
	a_k h_k = \sum_{\alpha} a_{k,\alpha} h_k,
	\]
	and by the graded multiplication, $a_{k,\alpha} h_k \in S_{\alpha+\delta_k}$. Thus each term $a_{k,\alpha} h_k$ is homogeneous and belongs to $I$ (since $h_k \in I$ and $I$ is an ideal). Hence $x$ is a sum of homogeneous elements of $I$, i.e., $x \in \bigoplus_{\gamma \in \Gamma} (I \cap S_\gamma)$. Thus $I = \bigoplus_{\gamma \in \Gamma} (I \cap S_\gamma)$.
\end{proof}

\begin{remark}
	Theorem \ref{homogeneouscomponents} and its proof are semiring analogues of the corresponding result in graded ring theory. See, for instance, Proposition 9.95 in \cite{Rotman2002}.
\end{remark}

For an $\mathbb{N}_0$-graded semiring $S = \bigoplus_{n \in \mathbb{N}_0} S_n$ and an ideal $A$ of $S$, we denote the component of degree $1$ by $A_1 = A \cap S_1$. Then $(A \cap B)_1 = A_1 \cap B_1$ for any ideals $A, B$ of $S$.

\begin{lemma}\label{componentintersection}
	Let $S = \bigoplus_{n \in \mathbb{N}_0} S_n$ be an $\mathbb{N}_0$-graded commutative semiring and let $I, J$ be ideals of $S$. Then $(I \cap J)_1 = I_1 \cap J_1$.
\end{lemma}

\begin{proof}
	By definition, $A_1 = A \cap S_1$ for any ideal $A$ of $S$. Hence
	\[
	(I \cap J)_1 = (I \cap J) \cap S_1 = (I \cap S_1) \cap (J \cap S_1) = I_1 \cap J_1. \qedhere
	\]
\end{proof}

\begin{theorem}\label{thm:non-nucleus}
	Let $S_0$ be a zerosumfree semiring and $S = S_0[X,Y,Z]$ be the polynomial semiring graded by total degree. Then $S$ is not a nucleus hemiring.
\end{theorem}

\begin{proof}
	Consider the homogeneous ideals $I = \langle Y, X+Y \rangle$ and $J = \langle Z, X+Z \rangle$ in $S$. We show that $\scl_S(I \cap J) \neq \scl_S(I) \cap \scl_S(J)$.
	
	First, observe that $X+Y \in I$ and $Y \in I$. Then $X \in \scl_S(I)$ because $X+Y$ and $Y$ are in the subtractive ideal $\scl_S(I)$. Analogously, since $X+Z \in J$ and $Z \in J$, we obtain $X \in \scl_S(J)$. Thus, $X \in \scl_S(I) \cap \scl_S(J)$.
	
	Next, we determine the intersection $(I \cap J)_1$. Since $I$ and $J$ are homogeneous, their intersection $I \cap J$ is also homogeneous by Theorem~\ref{homogeneouscomponents}(b). By Lemma~\ref{componentintersection}, $(I \cap J)_1 = I_1 \cap J_1$. The elements of $I_1$ and $J_1$ are formed via linear combinations of their homogeneous generators by constants in $S_0$:
	\[
	I_1 = \{\, bX + (a+b)Y \mid a,b \in S_0 \,\} \quad \text{and} \quad J_1 = \{\, dX + (c+d)Z \mid c,d \in S_0 \,\}.
	\]
	Equating elements $bX + (a+b)Y = dX + (c+d)Z$ forces $b=d$, $a+b=0$, and $c+d=0$ due to the linear independence of $X, Y, Z$ over $S_0$. Since $S_0$ is zerosumfree, $a+b=0$ implies $a=b=0$, and $c+d=0$ implies $c=d=0$. Hence, $I_1 \cap J_1 = \{0\}$, and consequently $(I \cap J)_1 = \{0\}$.
	
	Finally, suppose for contradiction that $X \in \scl_S(I \cap J)$. By Theorem~\ref{subtractiveclosureproperties}(a), there exists $u \in I \cap J$ such that $X+u \in I \cap J$. Since $I \cap J$ is homogeneous by Theorem~\ref{homogeneouscomponents}(b), taking degree-$1$ components yields $(X+u)_1 \in (I \cap J)_1$, so $X + u_1 = 0$. Since $X$ is a degree-$1$ variable and $u_1 \in S_1$, we write $u_1 = c_1 X + c_2 Y + c_3 Z$ for $c_i \in S_0$, yielding $(1+c_1)X + c_2 Y + c_3 Z = 0$. By the uniqueness of the monomial decomposition in $S$, comparing coefficients gives $1+c_1 = c_2 = c_3 = 0$, which contradicts the assumption that $1 \neq 0$ in $S_0$. Thus $X \notin \scl_S(I \cap J)$. However, $X \in \scl_S(I) \cap \scl_S(J)$, while $X \notin \scl_S(I \cap J)$. Therefore $\scl_S(I \cap J) \neq \scl_S(I) \cap \scl_S(J)$, and $S$ is not a nucleus hemiring.
\end{proof}

\section*{Open questions}

\smallskip \noindent (1) Does every uniserial hemiring satisfy multiplicative distributivity, i.e., \[\scl_H(I)\scl_H(J)=\scl_H(IJ)\] for all ideals $I,J$? 

\smallskip \noindent (2) Is the polynomial hemiring $S_0[X,Y,Z]$ over a zerosumfree semiring $S_0$ a Kuratowski hemiring?

\smallskip \noindent (3) Does the subtractive closure distribute over ideal multiplication in $S_0[X,Y,Z]$?

\end{document}